\documentclass{baustms}

\citesort

\usepackage{doi}
\usepackage{orcidlink}
\usepackage{tikz-cd}

\DeclareMathOperator{\Tor}{Tor}
\DeclareMathOperator{\supp}{supp}

\theoremstyle{cupthm}
\newtheorem{thm}{Theorem}[section]
\newtheorem{prop}[thm]{Proposition}
\newtheorem{cor}[thm]{Corollary}
\newtheorem{lemma}[thm]{Lemma}
\theoremstyle{cupdefn}
\newtheorem{defn}[thm]{Definition}
\newtheorem{nota}[thm]{Notation}
\newtheorem{example}[thm]{Example}
\numberwithin{equation}{section}

\makeatletter
\def\@cuplistvals{%
    \parsep \z@
    \leftmargin \z@
    \labelsep 0.4em
    \itemindent \z@
    \topsep 3\p@
    \partopsep \z@
    \itemsep \z@}
\def\@enum@{%
    \expandafter\list\csname label\@enumctr\endcsname
    {\usecounter{\@enumctr}%
    \@cuplistvals
    \settowidth\labelwidth{\normalfont 8.}%
    \itemindent\labelwidth
    \advance\itemindent\labelsep
    \def\makelabel##1{\hbox to\labelwidth{\normalfont##1\hss}}}}
\def\itemize{%
  \ifnum \@itemdepth >\thr@@ \@toodeep\else
    \advance\@itemdepth\@ne
    \edef\@itemitem{labelitem\romannumeral\the\@itemdepth}%
    \expandafter\list\csname\@itemitem\endcsname
      {\@cuplistvals
      \topsep \z@
      \settowidth\labelwidth{\csname\@itemitem\endcsname}%
      \itemindent\labelwidth
      \advance\itemindent\labelsep
      \def\makelabel##1{\hbox to\labelwidth{##1\hss}}}\fi}
\makeatother

\tikzset{
  wrapdown/.style={
    rounded corners=6pt,
    to path={-- ([xshift=2ex]\tikztostart.east)
             -- ([xshift=2ex]\tikztostart.east |- #1)
             -- ([xshift=-2ex]\tikztotarget.west |- #1) \tikztonodes
             -- ([xshift=-2ex]\tikztotarget.west)
             -- (\tikztotarget)}}}

\hypersetup{
  pdftitle={Discrete Coefficients and Open Invariant Covers in the Homology of Ample Groupoids},
  pdfauthor={Luciano Melodia},
  pdfsubject={Groupoid homology},
  pdfkeywords={Ample groupoid, Groupoid homology, Mayer-Vietoris sequence}}

\newcommand{\G}{\mathcal G}
\newcommand{\Z}{\mathbb Z}
\newcommand{\defeq}{\mathrel{\mathop:}=}
\newcommand{\xto}[1]{\xrightarrow{\ #1\ }}
\newcommand{\sto}[1]{\xrightarrow{#1}}

\begin{document}
\runningtitle{Coefficients and Invariant Covers}
\title{Discrete Coefficients and Open Invariant Covers in the Homology of Ample Groupoids}

\author[1]{Luciano Melodia\,\orcidlink{0000-0002-7584-7287}}
\address[1]{Department of Mathematics, Friedrich-Alexander-Universit\"at Erlangen--N\"urnberg, Cauerstra\ss e 11, 91058 Erlangen, Germany\email{luciano.melodia@fau.de}}

\authorheadline{L. Melodia}

\support{This paper is a writeup of part of the author's master's thesis \cite{MelodiaThesis}, which is available as a preprint. Earlier versions of the results of \S\ref{sec:coefficients} and \S\ref{sec:covers} appear there, and Theorem~\ref{thm:B} corrects the use that is made of one of them in that text. Theorem~\ref{thm:A} and the computation of \S\ref{sec:example} do not appear there at all.}

\begin{abstract}
The homology of an ample groupoid is computed from the complex of compactly supported continuous functions on the nerve. Two hypotheses routinely imposed on this complex behave in opposite ways. We show that the comparison map from the integral chain complex tensored with the coefficient group to the chain complex with coefficients is always injective, and that it is surjective if and only if every compactly supported continuous function into the coefficient group is locally constant. The universal coefficient theorem for ample groupoids is therefore a discrete coefficient statement, and it already fails for the real numbers on the Cantor set. We then show that a cover of the unit space by two clopen invariant subsets forces the groupoid to split as a disjoint union of three clopen invariant reductions, so that the associated Mayer--Vietoris sequence has vanishing connecting maps and computes nothing. Dropping closedness repairs this, since two open invariant subsets covering the unit space give a Mayer--Vietoris sequence for every topological abelian coefficient group. For an integer action on a two point compactification of the integers we compute that sequence and find that its connecting map is an isomorphism of infinite cyclic groups, which no cover of a unit space by clopen invariant subsets can achieve.
\end{abstract}


\maketitle

\section{Introduction}
\label{sec:intro}

Crainic and Moerdijk introduced homology groups for \'etale groupoids \cite{CrainicMoerdijk2000}. When the groupoid is ample, Matui gave a concrete model for them \cite[\S3.1]{Matui2012}: the chain group in degree \(n\) consists of the compactly supported continuous functions from the space of composable \(n\)-tuples into a topological abelian coefficient group, and the boundary is the alternating sum of the pushforwards along the face maps of the nerve. These groups are now a standard invariant of ample groupoids, and two hypotheses recur whenever one computes with them. Neither hypothesis is forced by the definition of the complex, and it is not obvious how much either of them is doing.

The first concerns coefficients. If the coefficient group \(A\) is discrete, then a compactly supported continuous function into \(A\) is locally constant with finite image, the degree \(n\) chain group is the integral chain group tensored with \(A\), and a universal coefficient theorem follows from the algebraic one for complexes of free abelian groups. This is a theorem of Miller and Steinberg \cite[Theorem~4.1]{MillerSteinberg2024}. If \(A\) carries a topology, the compactly supported continuous functions into \(A\) may form a strictly larger group than the span of those with finite image, and it is unclear how much of the theorem survives, or what it is that the discreteness of the coefficients supplies.

The second concerns cutting the unit space. One reduces an ample groupoid to invariant subsets of its unit space, and the reductions usually taken are to \emph{clopen} invariant subsets, since extension by zero along a clopen inclusion is transparent. A cover of the unit space by two clopen invariant subsets yields a short exact sequence of chain complexes, and hence a Mayer--Vietoris sequence. This is Theorem~3.3.10 of the author's master's thesis \cite{MelodiaThesis}, where the sequence is described as a tool for computing homology by cutting the unit space into pieces. What such a sequence computes, and whether it computes anything at all, is the question that Theorem~\ref{thm:B} below answers.

We settle both questions, and the answers point in opposite directions. Discreteness of the coefficients cannot be relaxed at all. Clopenness of the cover, by contrast, is too strong, since it forces the Mayer--Vietoris sequence to carry no information, and it must be weakened to openness before that sequence carries information which the direct sum decomposition of Theorem~\ref{thm:B} does not already contain, as the computation of \S\ref{sec:example} shows for an integer action on a compactification of the integers.

Our first result determines the range of the comparison map. Write \(C_c(X,A)\) for the group of compactly supported continuous maps from a space \(X\) to an abelian topological group \(A\), and write \(\chi_V\) for the characteristic function of a subset \(V\subseteq X\), which lies in \(C_c(X,\Z)\) as soon as \(V\) is compact and open in the space \(X\).

\begin{thm}
\label{thm:A}
Let \(X\) be a locally compact Hausdorff space with a basis of compact open sets, let \(A\) be an abelian topological group, and let
\[
\Phi_X\colon C_c(X,\Z)\otimes_{\Z}A\longrightarrow C_c(X,A),
\qquad
\Phi_X(f\otimes a)(x)=f(x)\,a ,
\]
be the canonical homomorphism. Then \(\Phi_X\) is injective, and the following three conditions on the pair \((X,A)\) are equivalent: \textup{(i)} every element of \(C_c(X,A)\) is locally constant, \textup{(ii)} the map \(\Phi_X\) is surjective, \textup{(iii)} the map \(\Phi_X\) is an isomorphism of abelian groups. In particular they hold whenever the group \(A\) is discrete.
\end{thm}

Discreteness of \(A\) is sufficient for these conditions but not necessary, and the conditions do fail: for the Cantor space and the real numbers with their usual topology, the map \(\Phi_X\) is injective and not surjective. Both points are settled in \S\ref{sec:coefficients}, where we also record the resulting universal coefficient theorem and delimit the scope that it has.

Our second result concerns clopen invariant covers. Call a subset \(U\) of the unit space of a groupoid \(\G\) invariant if \(r(\gamma)\in U\) for every arrow \(\gamma\) with \(s(\gamma)\in U\), and write \(\G|_U\) for the reduction of \(\G\) to \(U\). Both of these notions are recalled in \S\ref{sec:complex} below.

\begin{thm}
\label{thm:B}
Let \(\G\) be an ample groupoid, let \(A\) be an abelian topological group, and let \(U_1,U_2\) be clopen invariant subsets of the unit space of \(\G\) with \(U_1\cup U_2=\G^{(0)}\). Put \(W_1\defeq U_1\setminus U_2\), \(W_0\defeq U_1\cap U_2\) and \(W_2\defeq U_2\setminus U_1\). Then \(\G\) is the disjoint union of the clopen subgroupoids \(\G|_{W_1}\), \(\G|_{W_0}\) and \(\G|_{W_2}\), and for every degree \(n\geq 0\) the map \(\theta_n\) induced by extension by zero is an isomorphism
\[
H_n(\G|_{W_1};A)\oplus H_n(\G|_{W_0};A)\oplus H_n(\G|_{W_2};A)
\xto{\theta_n}
H_n(\G;A).
\]
The Mayer--Vietoris sequence of the cover \(U_1,U_2\) provided by Theorem~\ref{thm:C} below has vanishing connecting homomorphisms, and it splits into split short exact sequences in every degree \(n\geq0\), so that it carries no information beyond that decomposition.
\end{thm}

So a clopen invariant cover never relates a groupoid to genuinely smaller pieces. It only records that homology takes disjoint unions to direct sums. The hypothesis to drop is closedness and not invariance, because invariance is what makes the reductions cover the nerve, whereas closedness is what makes each of the three reductions split off as a direct summand of the whole Moore complex in every degree \(n\geq0\).

\begin{thm}
\label{thm:C}
Let \(\G\) be an ample groupoid, let \(A\) be an abelian topological group, and let \(U_1,U_2\) be open invariant subsets of the unit space of \(\G\) with \(U_1\cup U_2=\G^{(0)}\). Put \(U_0\defeq U_1\cap U_2\). Then there is a long exact sequence
\[
\begin{tikzcd}[column sep=3.0em, row sep=3.2em, arrow style=math font, cramped]
\cdots \arrow[r,"H_n(\alpha)"]
 & H_n(\G|_{U_1};A)\oplus H_n(\G|_{U_2};A) \arrow[r,"H_n(\beta)"]
   \arrow[d, phantom, ""{coordinate, name=Z}]
 & H_n(\G;A) \arrow[dll, wrapdown=Z, "{\textstyle\partial_n}"{description}] \\
H_{n-1}(\G|_{U_0};A) \arrow[r,"H_{n-1}(\alpha)"]
 & H_{n-1}(\G|_{U_1};A)\oplus H_{n-1}(\G|_{U_2};A) \arrow[r,"H_{n-1}(\beta)"] & \cdots
\end{tikzcd}
\]
in which \(H_n(\alpha)\) and \(H_n(\beta)\) are induced by the chain maps \(\alpha_\bullet\) and \(\beta_\bullet\) of extension by zero written out in Proposition~\ref{prop:ses}, and \(\partial_n\) is the connecting homomorphism of the short exact sequence of chain complexes that is constructed in that proposition.
\end{thm}

Theorem~\ref{thm:C} is proved in \S\ref{sec:covers} together with Theorem~\ref{thm:B}, and it holds for every topological abelian coefficient group, which is the point at which the Moore complex is more than a repackaging of the discrete theory. We should say at once that for a discrete coefficient group Theorem~\ref{thm:C} can also be obtained by applying the Barratt--Whitehead argument to two instances of the long exact sequence attached to a closed invariant subset of the unit space and its open complement \cite[Proposition~4.2]{MillerSteinberg2024}. The chain level proof given here is shorter, and it produces the connecting map explicitly at the level of the chain complexes and for topological coefficients.

The point of Theorem~\ref{thm:C} is that its connecting maps need not vanish, and our last result exhibits this behaviour in the smallest example that we have been able to find.

\begin{thm}
\label{thm:D}
Let \(X\) be the two point compactification of the integers, let \(T\colon X\to X\) be the shift, and let \(\G\) be the transformation groupoid of the action of the integers on \(X\) generated by \(T\). Let \(U_1\) and \(U_2\) be the complements in \(X\) of the two points at infinity. Then \(U_1\) and \(U_2\) are open and invariant, neither of them is closed, \(U_1\cup U_2=X\), and with integer coefficients one has
\[
H_1(\G|_{U_1})=H_1(\G|_{U_2})=H_1(\G|_{U_1\cap U_2})=0,
\qquad
H_1(\G)\cong\Z .
\]
The connecting homomorphism \(\partial_1\colon H_1(\G)\to H_0(\G|_{U_1\cap U_2})\) is an isomorphism of infinite cyclic groups, so that the sequence does not split into short exact sequences.
\end{thm}

Every piece of the cover has trivial first homology, the groupoid does not, and the Mayer--Vietoris sequence accounts for the whole of the difference. By Theorem~\ref{thm:B} no clopen invariant cover of any ample groupoid can ever behave in this way.

This paper is a writeup of part of the author's master's thesis \cite{MelodiaThesis}, and we state the relation to it in full. The thesis proves the universal coefficient theorem in the form of Corollary~\ref{cor:uct} and the Mayer--Vietoris sequence for clopen invariant covers, and it isolates the obstruction to the tensor comparison map for non-discrete coefficients. The universal coefficient theorem is due to Miller and Steinberg \cite[Theorem~4.1]{MillerSteinberg2024}, whose preprint predates the thesis by more than a year. The author obtained it independently and in ignorance of their work, and priority is theirs. The present paper cites their theorem rather than reproving it, and \S\ref{sec:coefficients} records how the two arguments differ. The Mayer--Vietoris sequence for clopen invariant covers is the statement that Theorem~\ref{thm:B} shows to be degenerate. What is new here is Theorem~\ref{thm:A}, which turns the obstruction into a criterion, together with Theorems~\ref{thm:B}, \ref{thm:C} and \ref{thm:D} as stated above.

The paper is organised as follows. In \S\ref{sec:complex} we fix the Moore complex and the reductions that we use, and record the two topological facts about compact open sets on which the later proofs rest. In \S\ref{sec:coefficients} we prove Theorem~\ref{thm:A}, deduce the universal coefficient theorem and delimit its scope by two examples. In \S\ref{sec:covers} we prove Theorems~\ref{thm:B} and \ref{thm:C}, and in \S\ref{sec:example} we prove Theorem~\ref{thm:D} by computing the homology of each of the four transformation groupoids that occur in its statement.

\section{The Moore Complex}
\label{sec:complex}

We first fix the objects and the notation for the whole paper. For background on \'etale groupoids and their bisections we refer to the lecture notes of Sims \cite[Ch.~2]{Sims2020}.

\begin{nota}
\label{nota:setup}
A topological groupoid \(\G\) has unit space \(\G^{(0)}\), range map \(r\) and source map \(s\). Arrows are written \(\gamma,\eta\) and units are written \(x,y,z\). We call \(\G\) \emph{\'etale} if \(s\) is a local homeomorphism, and \emph{ample} if in addition \(\G\) is locally compact and Hausdorff and \(\G^{(0)}\) has a basis of compact open sets. For \(n\geq 1\) we write
\[
\G^{(n)}\defeq\{(\gamma_1,\dots,\gamma_n)\in\G^n : s(\gamma_i)=r(\gamma_{i+1})\ \text{for}\ 1\leq i<n\}
\]
for the space of composable \(n\)-tuples, with the subspace topology from \(\G^n\). The face maps \(d_i\colon\G^{(n)}\to\G^{(n-1)}\) are given for \(n=1\) by \(d_0=s\) and \(d_1=r\), and for \(n\geq 2\) by
\[
d_i(\gamma_1,\dots,\gamma_n)=
\begin{cases}
(\gamma_2,\dots,\gamma_n), & i=0,\\
(\gamma_1,\dots,\gamma_i\gamma_{i+1},\dots,\gamma_n), & 1\leq i\leq n-1,\\
(\gamma_1,\dots,\gamma_{n-1}), & i=n.
\end{cases}
\]
For a locally compact Hausdorff space \(X\) and an abelian topological group \(A\), the set \(C_c(X,A)\) of continuous maps \(X\to A\) with compact support is an abelian group under pointwise addition, and \(\chi_V\in C_c(X,\Z)\) denotes the characteristic function of a compact open set \(V\subseteq X\). If \(\phi\colon X\to Y\) is a local homeomorphism between locally compact Hausdorff spaces, the pushforward
\(
\phi_*(f)(y)\defeq\sum_{\phi(x)=y}f(x)
\)
defines a homomorphism \(\phi_*\colon C_c(X,A)\to C_c(Y,A)\), and pushforward is functorial \cite[\S3.1]{Matui2012}. If \(\G\) is \'etale then each \(d_i\) is a local homeomorphism, and
\[
\delta_n\defeq\sum_{i=0}^{n}(-1)^i(d_i)_*\colon C_c(\G^{(n)},A)\to C_c(\G^{(n-1)},A),
\qquad n\geq1,
\]
together with \(\delta_0\defeq0\) makes \(C_\bullet(\G;A)\defeq(C_c(\G^{(n)},A),\delta_n)_{n\geq0}\) a chain complex, the \emph{Moore complex} of \(\G\) with coefficients in \(A\). We write \(H_n(\G;A)\) for its homology and \(H_n(\G)\) for \(H_n(\G;\Z)\). This is Definition~3.1 of \cite{Matui2012}, whose conventions we follow throughout. In degree one the boundary is the difference \(s_*-r_*\) of the two pushforwards along the source map and the range map, which are the face maps in that degree.
\end{nota}

Two elementary facts about compact open sets are used repeatedly, and we isolate them here rather than repeat them in each proof. Recall that a compact subset of a Hausdorff space is closed, so that a compact open subset of a Hausdorff space is clopen. Consequently the difference of two compact open subsets is again a compact open subset, and both of these facts are used below without further comment.

\begin{lemma}
\label{lem:basis}
Let \(\G\) be an ample groupoid and let \(n\geq0\). Then the space \(\G^{(n)}\) of composable \(n\)-tuples is locally compact and Hausdorff, and it has a basis of compact open sets. In particular each of its compact open subsets is clopen in it.
\end{lemma}

\begin{proof}~
\begin{enumerate}
\item\label{st:basis-one}
\textbf{The cases \(n=0\) and \(n=1\).} For \(n=0\) the claim holds by hypothesis. Let \(\gamma\in\G\) and let \(V\subseteq\G\) be open with \(\gamma\in V\). Since \(\G\) is \'etale there is an open set \(B\) with \(\gamma\in B\subseteq V\) such that \(s|_B\colon B\to s(B)\) is a homeomorphism onto an open subset of \(\G^{(0)}\). As \(\G^{(0)}\) has a basis of compact open sets, choose a compact open set \(K\subseteq s(B)\) with \(s(\gamma)\in K\). Then \((s|_B)^{-1}(K)\) is compact, is open in \(B\) and hence in \(\G\), and satisfies \(\gamma\in(s|_B)^{-1}(K)\subseteq V\), so that the compact open sets form a basis at \(\gamma\), as claimed for \(n=1\).

\item\label{st:basis-two}
\textbf{The case \(n\geq2\).} The maps \(s\) and \(r\) are continuous and \(\G\) is Hausdorff, so each condition \(s(\gamma_i)=r(\gamma_{i+1})\) is closed and \(\G^{(n)}\) is a closed subspace of \(\G^n\). By step~\ref{st:basis-one} the product \(\G^n\) is locally compact Hausdorff with a basis of compact open sets, namely the products of such sets. A closed subspace of a locally compact Hausdorff space is locally compact Hausdorff, and the intersection of a compact open set with a closed subspace is compact and open in that subspace. Hence \(\G^{(n)}\) has a basis of compact open sets, and each of them is clopen in \(\G^{(n)}\) because that space is Hausdorff.\qedhere
\end{enumerate}
\end{proof}

\begin{lemma}
\label{lem:split}
Let \(X\) be a locally compact Hausdorff space with a basis of compact open sets, let \(V_1,V_2\subseteq X\) be open, and let \(K\subseteq V_1\cup V_2\) be compact. Then there are disjoint compact open sets \(K_1,K_2\subseteq X\) with \(K_i\subseteq V_i\) for \(i=1,2\) and with \(K\subseteq K_1\cup K_2\).
\end{lemma}

\begin{proof}
Each \(x\in K\) lies in \(V_{i(x)}\) for some \(i(x)\in\{1,2\}\), so there is a compact open set \(B_x\) with \(x\in B_x\subseteq V_{i(x)}\). Finitely many of these cover \(K\), say \(B_{x_1},\dots,B_{x_p}\). Let \(K_1\) be the union of those \(B_{x_j}\) with \(i(x_j)=1\) and let \(K'\) be the union of the remaining ones, so that \(K_1\subseteq V_1\) and \(K'\subseteq V_2\) are compact open and \(K\subseteq K_1\cup K'\). Put \(K_2\defeq K'\setminus K_1\), which is compact open because \(K_1\) is clopen. Then \(K_2\subseteq V_2\), the two sets \(K_1\) and \(K_2\) are disjoint, and \(K\subseteq K_1\cup K'=K_1\cup K_2\), which is what the lemma asserts.
\end{proof}

Finally we fix the reductions that will be cut out of a groupoid. The \emph{orbit} of a unit \(x\) is the set \(r(s^{-1}(x))\) of the ranges of all arrows of \(\G\) with source \(x\).

\begin{defn}
\label{def:invariant}
Let \(\G\) be a topological groupoid. A subset \(U\subseteq\G^{(0)}\) is \emph{invariant} if \(r(\gamma)\in U\) for every \(\gamma\in\G\) with \(s(\gamma)\in U\). Equivalently, the set \(U\) is a union of orbits. The \emph{reduction} of \(\G\) to \(U\) is the subgroupoid \(\G|_U\defeq r^{-1}(U)\cap s^{-1}(U)\), which carries the subspace topology inherited from \(\G\) and has \(U\) as its own unit space.
\end{defn}

Invariant subsets are called saturated by some authors. We use invariant, which is the more common term in the groupoid literature and avoids a clash with the several other uses that the word saturated already has in topology and in algebra.

\begin{lemma}
\label{lem:reduction-ample}
Let \(\G\) be an ample groupoid and let \(U\subseteq\G^{(0)}\) be open, but neither necessarily closed nor necessarily invariant. Then the reduction \(\G|_U\) is again an ample groupoid, and its unit space is \(U\) with the topology inherited from \(\G^{(0)}\).
\end{lemma}

\begin{proof}
The sets \(r^{-1}(U)\) and \(s^{-1}(U)\) are open, so \(\G|_U\) is an open subgroupoid of \(\G\) and inherits the Hausdorff property. An open subspace of a locally compact Hausdorff space is locally compact, and the members of a basis of compact open sets that are contained in the subspace form such a basis for it. This applies to \(\G|_U\) by Lemma~\ref{lem:basis} and to \(U\subseteq\G^{(0)}\) by hypothesis. Restricting a local homeomorphism to an open subset gives a local homeomorphism, so that \(s|_{\G|_U}\) is one and \(\G|_U\) is \'etale as well.
\end{proof}

\section{Coefficients}
\label{sec:coefficients}

This section proves Theorem~\ref{thm:A}, records the universal coefficient theorem that follows from it, and shows by two examples that the condition in Theorem~\ref{thm:A} is strictly weaker than discreteness of \(A\) and that it does fail in practice.

We first record the finite disjoint refinement of a family of integer valued chains, on which both halves of the proof of that theorem are going to rest.

\begin{lemma}
\label{lem:refine}
Let \(X\) be a locally compact Hausdorff space with a basis of compact open sets and let \(f_1,\dots,f_m\in C_c(X,\Z)\). Then there are finitely many pairwise disjoint compact open sets \(V_1,\dots,V_r\subseteq X\) such that every one of the \(f_j\) is constant on each \(V_k\) and vanishes outside the union \(V_1\cup\dots\cup V_r\) of these finitely many sets.
\end{lemma}

\begin{proof}
Since \(\Z\) is discrete, each \(f_j\) is locally constant, so every \(x\in X\) has an open neighbourhood on which all of \(f_1,\dots,f_m\) are constant. Shrinking that neighbourhood, we obtain a compact open set \(B_x\) with the same property. The set \(K\defeq\bigcup_{j}\supp(f_j)\) is compact, so finitely many such sets cover it, say \(B_{x_1},\dots,B_{x_p}\). Put \(V_1\defeq B_{x_1}\) and \(V_k\defeq B_{x_k}\setminus(B_{x_1}\cup\dots\cup B_{x_{k-1}})\) for \(2\leq k\leq p\). Each \(V_k\) is compact open, the \(V_k\) are pairwise disjoint, their union is \(B_{x_1}\cup\dots\cup B_{x_p}\supseteq K\), and each \(f_j\) is constant on \(V_k\subseteq B_{x_k}\). Each \(f_j\) vanishes off \(K\) and hence off the union of the sets \(V_k\), because that union contains the compact set \(K\), and this proves the lemma as stated.
\end{proof}

\begin{proof}[Proof of Theorem~\ref{thm:A}]~
\begin{enumerate}
\item\label{st:inj}
\textbf{The map \(\Phi_X\) is injective.} Let \(\nu=\sum_{j=1}^{m}f_j\otimes a_j\) lie in the kernel of \(\Phi_X\). Choose \(V_1,\dots,V_r\) for \(f_1,\dots,f_m\) as in Lemma~\ref{lem:refine} and discard the empty ones. Writing \(f_j=\sum_{k=1}^{r}n_{jk}\chi_{V_k}\) with \(n_{jk}\in\Z\) and regrouping the resulting double sum gives
\[
\nu=\sum_{k=1}^{r}\chi_{V_k}\otimes b_k,
\qquad
b_k\defeq\sum_{j=1}^{m}n_{jk}a_j\in A,
\qquad
\Phi_X(\nu)=\sum_{k=1}^{r}b_k\,\chi_{V_k}.
\]
Evaluating the identity \(\Phi_X(\nu)=0\) at a point of the nonempty set \(V_k\) yields \(b_k=0\) for each index \(k\), and therefore \(\nu=0\) in the tensor product \(C_c(X,\Z)\otimes_{\Z}A\).

\item\label{st:i-ii}
\textbf{Condition \textup{(i)} implies condition \textup{(ii)}.} Let \(\xi\in C_c(X,A)\) be locally constant. The fibres of \(\xi\) are open and cover the compact set \(\supp(\xi)\), so the set \(F\defeq\xi(X)\setminus\{0\}\) of nonzero values is finite. For \(a\in F\) the fibre \(U_a\defeq\xi^{-1}(a)\) is open, and its complement is the union of the remaining fibres and hence open as well, so \(U_a\) is clopen. Being a closed subset of \(\supp(\xi)\), it is compact. The sets \(U_a\) for \(a\in F\) are pairwise disjoint compact open sets, and \(\xi=\sum_{a\in F}a\,\chi_{U_a}\), which is the image under \(\Phi_X\) of the element \(\sum_{a\in F}\chi_{U_a}\otimes a\) of the tensor product \(C_c(X,\Z)\otimes_{\Z}A\), so that \(\Phi_X\) is surjective.

\item\label{st:ii-i}
\textbf{Condition \textup{(ii)} implies condition \textup{(i)}.} Let \(\xi\in C_c(X,A)\) and write \(\xi=\Phi_X(\sum_{j=1}^{m}f_j\otimes a_j)\). Each \(f_j\) is locally constant because \(\Z\) is discrete, so every point of \(X\) has a neighbourhood \(W\) on which all the \(f_j\) are constant, and on \(W\) the map \(\xi\) takes the constant value \(\sum_j f_j(x)a_j\) computed at any point \(x\in W\), so that \(\xi\) is locally constant.

\item\label{st:equiv}
\textbf{The equivalence.} Condition \textup{(iii)} implies condition \textup{(ii)} trivially, and condition \textup{(ii)} implies condition \textup{(iii)} by step~\ref{st:inj}. Together with steps~\ref{st:i-ii} and \ref{st:ii-i} this closes the cycle of implications and makes the three conditions equivalent to one another.\qedhere
\end{enumerate}
\end{proof}

If \(A\) is discrete then every continuous map into \(A\) is locally constant, so Theorem~\ref{thm:A} applies. Combining this with Lemma~\ref{lem:basis} identifies the Moore complex with the integral one tensored by \(A\), and the universal coefficient theorem for complexes of free abelian groups \cite[Theorem~3A.3]{Hatcher2002} then gives the following statement, which is Theorem~4.1 of \cite{MillerSteinberg2024} in the formulation used there. We record it here because it is the statement whose scope the rest of the present section is going to delimit.

\begin{cor}
\label{cor:uct}
Let \(\G\) be an ample groupoid and let \(A\) be a discrete abelian group. Then for every degree \(n\geq0\) there is a short exact sequence
\[
0\longrightarrow H_n(\G)\otimes_{\Z}A
\xto{\kappa_n} H_n(\G;A)
\xto{\iota_n} \Tor_1^{\Z}\bigl(H_{n-1}(\G),A\bigr)\longrightarrow 0 ,
\]
in which \(\kappa_n\) and \(\iota_n\) are the two maps of the algebraic universal coefficient theorem for the complex \(C_\bullet(\G;\Z)\), transported along the isomorphism of Theorem~\ref{thm:A}, and where the group \(H_{-1}(\G)\) is understood to be zero, so that for \(n=0\) the map \(\kappa_0\) is an isomorphism onto \(H_0(\G;A)\) and the sequence carries no torsion term at all.
\end{cor}

\begin{proof}
By Lemma~\ref{lem:basis} and Theorem~\ref{thm:A} the maps \(\Phi_{\G^{(n)}}\) are isomorphisms, and they commute with the boundaries. Indeed, for a local homeomorphism \(\phi\colon X\to Y\), for \(f\in C_c(X,\Z)\), for \(a\in A\) and for \(y\in Y\),
\[
\Phi_Y(\phi_*(f)\otimes a)(y)=\Bigl(\sum_{\phi(x)=y}f(x)\Bigr)a
=\sum_{\phi(x)=y}\Phi_X(f\otimes a)(x)=\phi_*\bigl(\Phi_X(f\otimes a)\bigr)(y),
\]
and summing over the face maps with alternating signs gives \(\Phi_{\G^{(n-1)}}\circ(\delta_n\otimes\mathrm{id}_A)=\delta_n\circ\Phi_{\G^{(n)}}\), where the boundary on the left is that of \(C_\bullet(\G;\Z)\) and the one on the right is that of \(C_\bullet(\G;A)\). Each \(C_c(\G^{(n)},\Z)\) is free abelian: by Lemma~\ref{lem:refine} its elements have finite image and are in particular bounded, the group of all bounded integer valued functions on a set is free abelian by the theorem of N\"obeling \cite[Satz~1]{Nobeling1968}, and a subgroup of a free abelian group is free. Applying \cite[Theorem~3A.3]{Hatcher2002} to the complex \(C_\bullet(\G;\Z)\) and transporting the result along \(\Phi_\bullet\) now gives the asserted sequence.
\end{proof}

The two proofs of this statement take different routes, and the difference is what the rest of the paper turns on. In \cite{MillerSteinberg2024} the chain group in degree \(n\) with coefficients in \(A\) is defined as the span of the functions \(a\,\chi_V\) with \(V\) compact open, so that its identification with a tensor product is a statement about that span. Here the chain group consists of all compactly supported continuous functions into \(A\), and the identification comes from the finite disjoint refinement of Lemma~\ref{lem:refine}, while the freeness of the integral chain groups rests on N\"obeling's theorem in both treatments. Theorem~\ref{thm:A} says when the two chain groups agree, and Example~\ref{ex:cantor} shows that for a coefficient group carrying a nondiscrete topology the two of them can already differ in degree zero.

The condition of Theorem~\ref{thm:A} is strictly weaker than discreteness of the coefficient group, as the following example on a discrete locally compact space shows.

\begin{example}
\label{ex:positive}
Let \(X\) be the integers with the discrete topology and let \(A\) be the additive group of real numbers with its usual topology, which is not discrete. Every compact subset of \(X\) is finite, so every element of \(C_c(X,A)\) has finite support and is locally constant, every point of \(X\) being isolated in it. By Theorem~\ref{thm:A} the map \(\Phi_X\) is an isomorphism, even though the coefficient group \(A\) that we have taken here is not discrete, so that discreteness of \(A\) is not necessary for the conclusion.
\end{example}

The condition does fail, and it fails already on the space that motivates the whole subject, namely the Cantor space with the real numbers as coefficients.

\begin{example}
\label{ex:cantor}
Let \(X=\{0,1\}^{\mathbb N}\) be the Cantor space with the product topology, which is compact Hausdorff with a basis of compact open cylinder sets, and let \(A\) be the additive group of real numbers with its usual topology. Define \(\xi\colon X\to A\) by \(\xi(x)\defeq\sum_{n\geq1}2^{-n}x_n\). The series converges uniformly, so \(\xi\) is continuous, and it has compact support because \(X\) is compact. It is not locally constant: a nonempty open set \(V\subseteq X\) contains a cylinder \(Z\) that fixes the first \(N\) coordinates, and the two points of \(Z\) that vanish in all coordinates after the coordinate \(N+1\) and differ in that one coordinate have values differing by \(2^{-N-1}\neq0\). By Theorem~\ref{thm:A} the map \(\Phi_X\) is injective and fails to be surjective, so it is not an isomorphism and the criterion above is not vacuous.
\end{example}

Taking for \(\G\) the unit groupoid on the Cantor space \(X\), whose arrows are its units, Example~\ref{ex:cantor} shows that the identification of the Moore complex with the integral complex tensored by \(A\) fails for non-discrete \(A\) already in degree \(0\). For that groupoid every face map is the identity, so \(\delta_1=0\) and hence \(H_0(\G)=C(X,\Z)\) and \(H_0(\G;A)=C(X,A)\), while \(H_{-1}(\G)=0\). The map \(\kappa_0\) of Corollary~\ref{cor:uct} is then the map \(\Phi_X\) of Example~\ref{ex:cantor}, which is not surjective, so it is the conclusion of the universal coefficient theorem and not merely its proof that fails for the real numbers on the Cantor space. Corollary~\ref{cor:uct} is therefore a statement about discrete coefficients in a strong sense: its proof does not merely happen to use discreteness, but the chain level input that carries it is available under the condition of Theorem~\ref{thm:A} and under no weaker one, and that condition is implied by discreteness of the coefficient group without being exhausted by it, as the discrete integers with real coefficients already show in Example~\ref{ex:positive}, where the comparison map is an isomorphism.

\section{Invariant Covers}
\label{sec:covers}

Throughout this section invariance does the work of covering the nerve, and openness or closedness of the cover decides what the resulting sequence says. We begin with the covering statement, which uses invariance and nothing else.

\begin{lemma}
\label{lem:nerve}
Let \(\G\) be a topological groupoid and let \(U\subseteq\G^{(0)}\) be invariant. Then \((\G|_U)^{(0)}=U\), and for every \(n\geq1\),
\begin{equation}
\label{eq:nerve-fibre}
(\G|_U)^{(n)}=\{(\gamma_1,\dots,\gamma_n)\in\G^{(n)} : s(\gamma_n)\in U\}.
\end{equation}
Consequently \((\G|_U)^{(n)}\) is open in \(\G^{(n)}\) if \(U\) is open in \(\G^{(0)}\), and clopen if \(U\) is clopen, and the face maps satisfy \(d_i\bigl((\G|_U)^{(n)}\bigr)\subseteq(\G|_U)^{(n-1)}\) for all \(i\). If \(U'\subseteq\G^{(0)}\) is a second invariant subset, then
\begin{equation}
\label{eq:nerve-lattice}
(\G|_{U\cap U'})^{(n)}=(\G|_U)^{(n)}\cap(\G|_{U'})^{(n)},
\end{equation}
while \(U\cup U'=\G^{(0)}\) implies \(\G^{(n)}=(\G|_U)^{(n)}\cup(\G|_{U'})^{(n)}\) in every degree \(n\).
\end{lemma}

\begin{proof}~
\begin{enumerate}
\item\label{st:nerve-desc}
\textbf{The description \eqref{eq:nerve-fibre}.} If every \(\gamma_k\) lies in \(\G|_U\) then \(s(\gamma_n)\in U\). Conversely let \((\gamma_1,\dots,\gamma_n)\in\G^{(n)}\) with \(s(\gamma_n)\in U\). For each \(k\) the product \(\gamma_k\gamma_{k+1}\cdots\gamma_n\) is defined and has source \(s(\gamma_n)\) and range \(r(\gamma_k)\), so invariance of \(U\) gives \(r(\gamma_k)\in U\). For \(k<n\) we have \(s(\gamma_k)=r(\gamma_{k+1})\in U\), and \(s(\gamma_n)\in U\) by assumption. Hence each \(\gamma_k\) lies in \(r^{-1}(U)\cap s^{-1}(U)=\G|_U\). The case \(n=0\) is the definition of the reduction.

\item\label{st:nerve-top}
\textbf{Topology.} For \(n\geq1\) let \(\sigma_n\colon\G^{(n)}\to\G^{(0)}\) be the map \(\sigma_n(\gamma_1,\dots,\gamma_n)\defeq s(\gamma_n)\), which is continuous because \(s\) is, and put \(\sigma_0\defeq\mathrm{id}_{\G^{(0)}}\). By \eqref{eq:nerve-fibre} the set \((\G|_U)^{(n)}\) is the preimage \(\sigma_n^{-1}(U)\), so that it is open whenever \(U\) is open and clopen whenever \(U\) is clopen in \(\G^{(0)}\).

\item\label{st:nerve-faces}
\textbf{Face maps.} Let \((\gamma_1,\dots,\gamma_n)\in(\G|_U)^{(n)}\), so that \(s(\gamma_n)\in U\). For \(n=1\) both \(d_0=s\) and \(d_1=r\) land in \(U\) by invariance. For \(n\geq2\) and \(i<n\) the last entry of \(d_i(\gamma_1,\dots,\gamma_n)\) is \(\gamma_n\) or \(\gamma_{n-1}\gamma_n\), and in both cases its source is \(s(\gamma_n)\in U\). For \(i=n\) the last entry is \(\gamma_{n-1}\), and \(s(\gamma_{n-1})=r(\gamma_n)\in U\) by invariance. In each case the source of the last entry of the image lies in \(U\), so that the description \eqref{eq:nerve-fibre} applies to the image tuple as well.

\item\label{st:nerve-lattice}
\textbf{Intersections and unions.} Both assertions follow from \eqref{eq:nerve-fibre}, since \(s(\gamma_n)\) lies in \(U\cap U'\) when and only when it lies in both, and lies in \(U\cup U'=\G^{(0)}\) always.\qedhere
\end{enumerate}
\end{proof}

We now prove that a clopen invariant cover splits the groupoid. The proof uses Theorem~\ref{thm:C} and Proposition~\ref{prop:ses}, both of which are proved below and neither of which depends on the statement of Theorem~\ref{thm:B}, so that the argument is not circular.

\begin{proof}[Proof of Theorem~\ref{thm:B}]~
\begin{enumerate}
\item\label{st:B-pieces}
\textbf{The three pieces.} The sets \(W_1=U_1\setminus U_2\), \(W_0=U_1\cap U_2\) and \(W_2=U_2\setminus U_1\) are pairwise disjoint with union \(U_1\cup U_2=\G^{(0)}\). Each is clopen, being obtained from the clopen sets \(U_1,U_2\) by intersection and complement, and each is invariant, being a union of orbits because \(U_1\) and \(U_2\) are. By Lemma~\ref{lem:reduction-ample} each of the three reductions \(\G|_{W_j}\) is an ample groupoid, and its unit space is the clopen set \(W_j\).

\item\label{st:B-nerve}
\textbf{The nerve splits.} Fix \(n\geq0\). By \eqref{eq:nerve-fibre} the sets \((\G|_{W_j})^{(n)}\) for \(j\in\{1,0,2\}\) are the preimages of \(W_1,W_0,W_2\) under the map \(\sigma_n\) of step~\ref{st:nerve-top} in the proof of Lemma~\ref{lem:nerve}, so they are pairwise disjoint clopen subsets of \(\G^{(n)}\) with union \(\G^{(n)}\). In particular the groupoid \(\G=\G^{(1)}\) is the disjoint union of the three clopen subgroupoids \(\G|_{W_j}\).

\item\label{st:B-complex}
\textbf{The complex splits.} Since the three pieces are clopen and partition \(\G^{(n)}\), restriction and extension by zero identify
\[
C_c(\G^{(n)},A)\cong\bigoplus_{j\in\{1,0,2\}}C_c\bigl((\G|_{W_j})^{(n)},A\bigr).
\]
Indeed, a compactly supported continuous function is the sum of its restrictions to the pieces, each of which is continuous and compactly supported because the pieces are clopen, and the decomposition is unique. By step~\ref{st:nerve-faces} in the proof of Lemma~\ref{lem:nerve} each face map preserves each piece, so the identification is one of chain complexes, and the isomorphism \(\theta_n\) in homology follows because the homology of a finite direct sum of complexes is the direct sum of the homology groups of the summands.

\item\label{st:B-degen}
\textbf{Degeneracy of the Mayer--Vietoris sequence.} Clopen invariant subsets are in particular open invariant, so Theorem~\ref{thm:C} applies to the cover \(U_1,U_2\). Its first map is induced by the chain map \(\alpha_n\) of Proposition~\ref{prop:ses}, and by step~\ref{st:B-complex} applied to \(U_1\) and to \(U_2\) we may write
\[
C_c\bigl((\G|_{U_1})^{(n)},A\bigr)\cong C_c\bigl((\G|_{W_1})^{(n)},A\bigr)\oplus C_c\bigl((\G|_{W_0})^{(n)},A\bigr),
\]
and similarly for \(U_2\) with summands indexed by \(W_0\) and \(W_2\). Under these identifications the map \(\alpha_n\) sends \(\xi\) to \(\bigl((0,\xi),(-\xi,0)\bigr)\). Hence \(\alpha_\bullet\) is a split monomorphism of chain complexes, with retraction the projection onto the summand indexed by \(W_0\) in the first coordinate, so the induced maps in homology are injective in every degree. Exactness of the sequence of Theorem~\ref{thm:C} then forces \(\partial_n=0\) for all \(n\), and the sequence breaks into short exact sequences, which are split by that same retraction.\qedhere
\end{enumerate}
\end{proof}

We turn to open invariant covers. The following is the chain level statement behind Theorem~\ref{thm:C}, and it is in its proof that Lemma~\ref{lem:split}, and with it the existence of a basis of compact open sets rather than mere \'etaleness, does the essential work.

\begin{prop}
\label{prop:ses}
Let \(\G\) be an ample groupoid, let \(A\) be an abelian topological group, and let \(U_1,U_2\subseteq\G^{(0)}\) be open invariant subsets with \(U_1\cup U_2=\G^{(0)}\). Put \(U_0\defeq U_1\cap U_2\). For \(n\geq0\) define
\[
\alpha_n(\xi)\defeq\bigl(\xi^{(1)},-\xi^{(2)}\bigr),
\qquad
\beta_n(\xi_1,\xi_2)\defeq\widetilde\xi_1+\widetilde\xi_2,
\]
where \(\xi^{(i)}\) is the extension by zero of \(\xi\in C_c((\G|_{U_0})^{(n)},A)\) to \((\G|_{U_i})^{(n)}\) and \(\widetilde\xi_i\) is the extension by zero of \(\xi_i\in C_c((\G|_{U_i})^{(n)},A)\) to \(\G^{(n)}\). Then
\[
0\longrightarrow C_\bullet(\G|_{U_0};A)\xto{\alpha_\bullet}
C_\bullet(\G|_{U_1};A)\oplus C_\bullet(\G|_{U_2};A)\xto{\beta_\bullet}
C_\bullet(\G;A)\longrightarrow 0
\]
is a short exact sequence of chain complexes of abelian groups and homomorphisms.
\end{prop}

\begin{proof}
Fix \(n\geq0\). By Lemma~\ref{lem:nerve} the sets \((\G|_{U_1})^{(n)}\) and \((\G|_{U_2})^{(n)}\) are open in \(\G^{(n)}\), they cover \(\G^{(n)}\), and their intersection is \((\G|_{U_0})^{(n)}\). Extension by zero along an open inclusion of locally compact Hausdorff spaces carries \(C_c\) to \(C_c\), so \(\alpha_n\) and \(\beta_n\) are defined on the groups in question.
\begin{enumerate}
\item\label{st:ses-inj}
\textbf{The map \(\alpha_n\) is injective and \(\beta_n\alpha_n=0\).} Extension by zero is injective, so \(\alpha_n\) is injective as well. For \(\xi\in C_c((\G|_{U_0})^{(n)},A)\) both \(\widetilde{\xi^{(1)}}\) and \(\widetilde{\xi^{(2)}}\) equal the extension by zero of \(\xi\) to \(\G^{(n)}\), so the two terms cancel in the sum that defines \(\beta_n\), and therefore the composite \(\beta_n\alpha_n\) satisfies \(\beta_n(\alpha_n(\xi))=0\) for every element \(\xi\) of that chain group.

\item\label{st:ses-ker}
\textbf{The kernel of \(\beta_n\).} Let \(\beta_n(\xi_1,\xi_2)=0\), so that \(\widetilde\xi_1=-\widetilde\xi_2\) on \(\G^{(n)}\). The function \(\widetilde\xi_1\) vanishes off \((\G|_{U_1})^{(n)}\), and it vanishes off \((\G|_{U_2})^{(n)}\) because \(\widetilde\xi_2\) does. Hence its support is contained in \(\supp(\widetilde\xi_1)\cap\supp(\widetilde\xi_2)\), which is a compact subset of \((\G|_{U_1})^{(n)}\cap(\G|_{U_2})^{(n)}=(\G|_{U_0})^{(n)}\). Therefore \(\xi\defeq\widetilde\xi_1|_{(\G|_{U_0})^{(n)}}\) lies in \(C_c((\G|_{U_0})^{(n)},A)\), and \(\xi^{(1)}=\xi_1\) and \(-\xi^{(2)}=\xi_2\) by construction, so the pair \((\xi_1,\xi_2)=\alpha_n(\xi)\) lies in the image of the map \(\alpha_n\), and the kernel of \(\beta_n\) is contained in that image and hence equal to it by step~\ref{st:ses-inj}.

\item\label{st:ses-surj}
\textbf{The map \(\beta_n\) is surjective.} Let \(\eta\in C_c(\G^{(n)},A)\). By Lemma~\ref{lem:basis} the space \(\G^{(n)}\) is locally compact Hausdorff with a basis of compact open sets, so Lemma~\ref{lem:split} applied to the compact set \(\supp(\eta)\) and the open sets \((\G|_{U_1})^{(n)},(\G|_{U_2})^{(n)}\) gives disjoint compact open sets \(K_i\subseteq(\G|_{U_i})^{(n)}\) with \(\supp(\eta)\subseteq K_1\cup K_2\). Each \(K_i\) is clopen, so \(\eta\chi_{K_i}\) is continuous, and its support is contained in \(K_i\) and hence compact. Let \(\eta_i\) be its restriction to \((\G|_{U_i})^{(n)}\). Then \(\beta_n(\eta_1,\eta_2)=\eta\chi_{K_1}+\eta\chi_{K_2}=\eta\), the last equality holding because \(\eta\) vanishes off \(K_1\cup K_2\), a compact open set that contains its whole support.

\item\label{st:ses-chain}
\textbf{Chain maps.} By step~\ref{st:nerve-faces} in the proof of Lemma~\ref{lem:nerve} the face maps restrict to each \((\G|_{U_i})^{(n)}\), so the boundaries of the three complexes are defined. Extension by zero commutes with pushforward along a face map \(d\colon\G^{(n)}\to\G^{(n-1)}\). Indeed, for \(f\in C_c((\G|_{U})^{(n)},A)\) with \(U\) invariant and for \(y\in\G^{(n-1)}\), the sum \(\sum_{d(\gamma)=y}\widetilde f(\gamma)\) runs effectively over those \(\gamma\in(\G|_U)^{(n)}\) with \(d(\gamma)=y\). If \(y\notin(\G|_U)^{(n-1)}\) there are none, since \(d\) maps \((\G|_U)^{(n)}\) into \((\G|_U)^{(n-1)}\), and the sum is \(0\). If \(y\in(\G|_U)^{(n-1)}\) the sum is the pushforward of \(f\) along the restricted face map. Summing over the face maps with alternating signs shows that \(\alpha_\bullet\) and \(\beta_\bullet\) commute with the boundaries.\qedhere
\end{enumerate}
\end{proof}

\begin{proof}[Proof of Theorem~\ref{thm:C}]
Apply the long exact homology sequence of a short exact sequence of chain complexes \cite[Theorem~1.3.1]{Weibel1994} to the sequence of Proposition~\ref{prop:ses}, using that homology takes the direct sum of two complexes to the direct sum of their homology groups, which identifies the middle term of that sequence.
\end{proof}

The connecting map is the usual one, and we record it here for use in \S\ref{sec:example}. Let \([c]\in H_n(\G;A)\) be represented by a cycle \(c\). Choose \(b\) with \(\beta_n(b)=c\), which is possible by Proposition~\ref{prop:ses}. Then \(\beta_{n-1}(\delta_n b)=\delta_n(\beta_n b)=\delta_n c=0\), so there is a unique \(a\) with \(\alpha_{n-1}(a)=\delta_n b\), and \(a\) is a cycle: for \(n\geq2\) because \(\alpha_{n-2}(\delta_{n-1}a)=\delta_{n-1}\delta_n b=0\) and \(\alpha_{n-2}\) is injective, and for \(n=1\) because \(\delta_0=0\) makes every \(0\)-chain a cycle. The class \(\partial_n([c])\) is by definition the class \([a]\) of the cycle \(a\) that is so constructed.

\section{A Cover with Nonvanishing Connecting Map}
\label{sec:example}

We now prove Theorem~\ref{thm:D}. The homology of the four groupoids involved is computed by the following lemma, which collapses Matui's spectral sequence for a semidirect product of an \'etale groupoid by a countable discrete group.

\begin{lemma}
\label{lem:transformation}
Let \(Y\) be a locally compact Hausdorff space with a basis of compact open sets, let \(T\colon Y\to Y\) be a homeomorphism, and let \(A\) be an abelian topological group. Let \(\G\) be the transformation groupoid \(\Z\ltimes_T Y\), that is, the space \(\Z\times Y\) with \(r(k,y)=T^k y\), with \(s(k,y)=y\) and with \((k,T^{l}y)(l,y)=(k+l,y)\). Write \(M\defeq C_c(Y,A)\) and let \(t\) be the automorphism \(T_*\) of \(M\). Then
\[
H_0(\G;A)\cong M/(t-1)M,
\qquad
H_1(\G;A)\cong\ker(t-1),
\qquad
H_n(\G;A)=0\ \text{for}\ n\geq2 .
\]
The isomorphism in degree zero is induced by the identity map of \(M\), so that it is compatible with extension by zero along an inclusion of open invariant subsets of \(Y\).
\end{lemma}

\begin{proof}~
\begin{enumerate}
\item\label{st:tr-unit}
\textbf{The unit groupoid.} Let \(\mathcal Y\) be the groupoid whose arrows are the points of \(Y\) and whose structure maps are all the identity. It is \'etale, and \(\mathcal Y^{(n)}=Y\) with every face map the identity, so \(\delta_n=\sum_{i=0}^{n}(-1)^i\mathrm{id}_M\), which is \(0\) for odd \(n\) and \(\mathrm{id}_M\) for even \(n\geq2\). Hence \(H_0(\mathcal Y;A)=M\), and for odd \(n\) we get \(H_n=M/M=0\), while for even \(n\geq2\) we get \(H_n=0\) as well, because in that case the boundary \(\delta_n\) is injective on \(M\).

\item\label{st:tr-spectral}
\textbf{The spectral sequence.} The homeomorphism \(T\) generates an action of \(\Z\) on \(\mathcal Y\) whose semidirect product is \(\G\). By \cite[Theorem~3.8(2)]{Matui2012} there is a spectral sequence \(E^2_{p,q}=H_p(\Z,H_q(\mathcal Y;A))\Rightarrow H_{p+q}(\G;A)\), where \(\Z\) acts on \(H_q(\mathcal Y;A)\) through the action. By step~\ref{st:tr-unit} the row \(q=0\) is the only nonzero one, so the spectral sequence collapses and \(H_n(\G;A)\cong H_n(\Z,M)\), the module structure on \(M=H_0(\mathcal Y;A)\) being the one that is induced by the action of \(T\), namely the automorphism \(t\) of \(M\) that is given by pushforward of compactly supported functions along the homeomorphism \(T\).

\item\label{st:tr-cyclic}
\textbf{Homology of the infinite cyclic group.} Identify the group ring of \(\Z\) with the ring \(R\) of Laurent polynomials in \(t\), and let \(\varepsilon\colon R\to\Z\) be the augmentation. Then
\(
0\longrightarrow R\xto{t-1}R\xto{\varepsilon}\Z\longrightarrow0
\)
is a free resolution of the trivial module \(\Z\), so that \(H_n(\Z,M)=\Tor^R_n(\Z,M)\) is \(M/(t-1)M\) for \(n=0\), is \(\ker(t-1)\) for \(n=1\) and vanishes for \(n\geq2\) \cite[Example~6.1.4]{Weibel1994}. The conclusion does not depend on whether one lets \(t\) act as \(T_*\) or as its inverse, since \(t-1\) and \(t^{-1}-1=-t^{-1}(t-1)\) differ by the unit \(-t^{-1}\) of \(R\) and therefore have the same kernel and the same image in the module \(M\).

\item\label{st:tr-degzero}
\textbf{Degree zero.} The last assertion is checked on chains and needs no spectral sequence. A \(1\)-chain of \(\G\) is a finite sum \(\sum_k g_k\) in which \(g_k\) is supported on \(\{k\}\times Y\) and is given by \(g_k(k,y)=f_k(y)\) for some \(f_k\in M\). From \(s(k,y)=y\) and \(r(k,y)=T^ky\) we get \(s_*(g_k)=f_k\) and \(r_*(g_k)=t^kf_k\), so that \(\delta_1(\sum_k g_k)=\sum_k(1-t^k)f_k\). Every \(1-t^k\) is a multiple of \(1-t\) in \(R\) and \(1-t\) occurs for \(k=1\), so the image of \(\delta_1\) is \((t-1)M\) and \(H_0(\G;A)\) is the quotient of \(M\) by it. Extension by zero of \(0\)-chains along an inclusion of open invariant subsets is compatible with this description, since the description uses no map other than the identity of \(M\) and the projection onto the quotient.\qedhere
\end{enumerate}
\end{proof}

\begin{proof}[Proof of Theorem~\ref{thm:D}]~
\begin{enumerate}
\item\label{st:D-space}
\textbf{The space.} Let \(X\defeq\Z\cup\{-\infty,+\infty\}\), and for \(N\in\Z\) put
\[
L_N\defeq\{-\infty\}\cup\{m\in\Z:m\leq N\},
\qquad
R_N\defeq\{+\infty\}\cup\{m\in\Z:m\geq N\}.
\]
Topologise \(X\) by the basis consisting of the singletons \(\{m\}\) with \(m\in\Z\) together with all \(L_N\) and all \(R_N\). This family is a basis because \(L_N\cap R_M\) is a finite set of integers, \(L_N\cap L_M=L_{\min(N,M)}\) and \(R_N\cap R_M=R_{\max(N,M)}\). Each \(L_N\) is open with open complement \(R_{N+1}\), hence clopen, and it is compact: a basic open set containing \(-\infty\) is some \(L_K\), and \(L_N\setminus L_K\) is finite. Symmetrically each \(R_N\) is compact open, and \(X\) itself is compact. Distinct integers are separated by singletons, an integer \(m\) and \(-\infty\) by \(\{m\}\) and \(L_{m-1}\), and \(-\infty\) and \(+\infty\) by \(L_0\) and \(R_1\), so \(X\) is Hausdorff. Thus \(X\) is compact and Hausdorff, and the sets listed above form a basis of the topology of \(X\) that consists of compact open sets, so that the space \(X\) satisfies the hypotheses of Lemma~\ref{lem:transformation} above.

\item\label{st:D-cover}
\textbf{The groupoid and the cover.} Let \(T(m)\defeq m+1\) for \(m\in\Z\) and \(T(\pm\infty)\defeq\pm\infty\). Then \(T(L_N)=L_{N+1}\), \(T(R_N)=R_{N+1}\) and \(T(\{m\})=\{m+1\}\), so \(T\) is a homeomorphism, and \(\G\defeq\Z\ltimes_T X\) is an ample groupoid. Its orbits are \(\Z\), \(\{-\infty\}\) and \(\{+\infty\}\), so the invariant subsets are the unions of these. Put \(U_1\defeq X\setminus\{-\infty\}\) and \(U_2\defeq X\setminus\{+\infty\}\). Both are invariant, both are open because \(U_1=R_0\cup\bigcup_{m<0}\{m\}\) and symmetrically for \(U_2\), and neither is closed because no basic open set containing a point at infinity misses \(\Z\). Moreover \(U_1\cup U_2=X\) and \(U_0\defeq U_1\cap U_2=\Z\). Each reduction of \(\G\) to an invariant subset \(U\) is the transformation groupoid \(\Z\ltimes_T U\) of the restricted action.

\item\label{st:D-modules}
\textbf{The four chain modules.} We claim that, with integer coefficients and with \(t=T_*\),
\[
C(X,\Z)\cong\Z\oplus R,
\qquad
C_c(U_1,\Z)\cong R,
\qquad
C_c(U_2,\Z)\cong R,
\qquad
C_c(U_0,\Z)\cong R,
\]
as modules over the ring \(R\) of Laurent polynomials in \(t\). Here the summand \(\Z\) of the first module is generated by \(\chi_X\) and carries the trivial action, while the four copies of \(R\) are free of rank one on \(\chi_{L_0}\), on \(\chi_{R_0}\), on \(\chi_{L_0}\) and on the characteristic function \(\chi_{\{0\}}\) of the singleton \(\{0\}\), respectively, the first \(\chi_{L_0}\) being taken in \(C(X,\Z)\).

For \(U_2\), note first that \(t\chi_{L_N}=\chi_{L_{N+1}}\). The functions \(\chi_{L_N}\) with \(N\in\Z\) span \(C_c(U_2,\Z)\). Indeed, given \(f\), continuity into the discrete group \(\Z\) makes \(f\) locally constant, so \(f\) is constant with some value \(a\) on some \(L_N\). Compactness of \(\supp(f)\) makes \(\supp(f)\) bounded above, say \(f(m)=0\) for \(m\geq M>N\). Then \(f-a\chi_{L_N}\) is supported in the finite set \(\{N+1,\dots,M-1\}\), and \(\chi_{\{m\}}=\chi_{L_m}-\chi_{L_{m-1}}\). They are independent: if \(\sum_N a_N\chi_{L_N}=0\) with finitely many \(a_N\) nonzero, then evaluating at an integer \(m\) gives \(\sum_{N\geq m}a_N=0\) for every \(m\), and subtracting consecutive relations gives \(a_N=0\) for every \(N\). Hence \(C_c(U_2,\Z)\) is free of rank one over \(R\) on \(\chi_{L_0}\). The case of \(U_1\) is symmetric, with \(R_N\) in place of \(L_N\), and the case of \(U_0=\Z\) is immediate from \(t\chi_{\{m\}}=\chi_{\{m+1\}}\). For \(X\) the same argument, applied to a function that is eventually constant at both ends, shows that \(\chi_X\) together with the \(\chi_{L_N}\) is a basis. Independence follows by evaluating first at \(+\infty\), which kills every \(\chi_{L_N}\), and then as before. Since \(t\chi_X=\chi_X\), the claim follows in all four cases.

\item\label{st:D-homology}
\textbf{The homology groups.} Apply Lemma~\ref{lem:transformation}. On a free \(R\)-module of rank one the map \(t-1\) is injective with cokernel \(\Z\), so
\[
H_0(\G|_{U_1})\cong H_0(\G|_{U_2})\cong H_0(\G|_{U_0})\cong\Z,
\qquad
H_1(\G|_{U_1})=H_1(\G|_{U_2})=H_1(\G|_{U_0})=0 ,
\]
and on \(\Z\oplus R\) the map \(t-1\) has kernel \(\Z\chi_X\) and cokernel \(\Z\oplus\Z\), so that \(H_0(\G)\cong\Z^2\) and \(H_1(\G)\cong\Z\), the latter corresponding under Lemma~\ref{lem:transformation} to the generator \(\chi_X\) of \(\ker(t-1)\). All the homology in degrees \(n\geq2\) vanishes, again by Lemma~\ref{lem:transformation}.

\item\label{st:D-connecting}
\textbf{The connecting map.} By step~\ref{st:D-cover} the cover \(U_1,U_2\) satisfies the hypotheses of Theorem~\ref{thm:C}, which in low degrees gives the exact sequence
\[
0\to H_1(\G)\sto{\partial_1}H_0(\G|_{U_0})
\sto{H_0(\alpha)} H_0(\G|_{U_1})\oplus H_0(\G|_{U_2})
\sto{H_0(\beta)} H_0(\G)\to0 ,
\]
the initial zero because \(H_1(\G|_{U_1})\oplus H_1(\G|_{U_2})=0\) by step~\ref{st:D-homology}. So \(\partial_1\) is injective. The map \(H_0(\alpha)\) sends the generator \([\chi_{\{0\}}]\) of \(H_0(\G|_{U_0})\cong\Z\) to the classes of the extensions by zero of \(\chi_{\{0\}}\). In \(H_0(\G|_{U_2})\) we have \(\chi_{\{0\}}=\chi_{L_0}-\chi_{L_{-1}}=(t-1)\chi_{L_{-1}}\), so that class vanishes, and in \(H_0(\G|_{U_1})\) we have \(\chi_{\{0\}}=\chi_{R_0}-\chi_{R_1}=-(t-1)\chi_{R_0}\), so that class vanishes as well. Hence \(H_0(\alpha)\) is zero, and exactness makes \(\partial_1\) surjective. Being injective and surjective between two infinite cyclic groups, the map \(\partial_1\) is an isomorphism of groups.\qedhere
\end{enumerate}
\end{proof}

The computation locates the phenomenon. By step~\ref{st:D-homology} the group \(H_1(\G)\cong\Z\) is carried by Lemma~\ref{lem:transformation} to the kernel of \(t-1\) on \(C(X,\Z)\), and a function in that kernel is invariant, hence constant on the dense orbit \(\Z\), and then constant on all of \(X\) by continuity at the two points at infinity. Such a function has compact support because \(X\) is compact, and it is the removal of either point at infinity that destroys compactness of the unit space and with it the whole of the first homology. Neither point at infinity lies in the intersection \(U_0\), and each of them lies in one of the sets \(U_1\) and \(U_2\) but not in the other. Removing a point that fails to be isolated is the operation that a clopen invariant cover cannot perform, and it is what the cover above does at each of the two points at infinity. By Theorem~\ref{thm:B} no clopen invariant cover produces a connecting homomorphism that differs from zero, whatever the ample groupoid may be and whatever the coefficient group in which the homology is taken.

\ack
The author's master's thesis \cite{MelodiaThesis}, from which the present paper is drawn, was written at Friedrich-Alexander-Universit\"at Erlangen--N\"urnberg under the supervision of Kang Li and Cath\'erine Meusburger, whom the author thanks. Theorem~\ref{thm:B} corrects the reading of Theorem~3.3.10 of that thesis: the Mayer--Vietoris sequence stated there for clopen invariant covers is exact, as claimed, but it is degenerate and is not a tool for computation. The author also thanks the anonymous referee of an earlier version, who drew attention to \cite{MillerSteinberg2024} and whose objections to the scope of both main results led to the present paper, and in particular to Theorem~\ref{thm:B} and to the computation carried out in \S\ref{sec:example}, neither of which was part of the version of the paper that the referee saw.

\bibliographystyle{srtnumbered}
\bibliography{main}

\end{document}